\documentclass[a4paper]{amsart}

\usepackage[a4paper, bottom=3cm, left=3cm, right=3cm]{geometry}
\usepackage{amssymb}
\usepackage[all]{xy}
\usepackage{enumerate}
\usepackage{wasysym}
\usepackage{graphicx,color}

\usepackage{datetime}

\usepackage{mathtools}

\newcommand{\harm}{\mathcal{H}}

\DeclareMathOperator{\Imm}{Im}
\DeclareMathOperator{\Id}{Id}
\newcommand{\R}{\mathbb R}

\newcommand{\vect}{\mathfrak{X}}

\newcommand{\lie}{\mathcal{L}}                             

\newcommand{\eps}{\epsilon}
\newcommand{\epseta}{\eps_\eta}
\newcommand{\id}{\mathrm{Id}}

\newcommand{\fol}{\mathcal{F}}
\newcommand{\folw}{\mathcal{\tilde F}}
\newcommand*{\longhookrightarrow}{\ensuremath{\lhook\joinrel\relbar\joinrel\rightarrow}}

\DeclareMathOperator{\lef}{Lef}

\newtheorem{theorem}{Theorem}[section]
\newtheorem{lemma}[theorem]{Lemma}
\newtheorem{corollary}[theorem]{Corollary}
\newtheorem{proposition}[theorem]{Proposition}
\theoremstyle{definition}

\theoremstyle{remark}
\newtheorem{remark}[theorem]{Remark}

\numberwithin{equation}{section}

\title{Hard Lefschetz Theorem for Vaisman manifolds}

\author[B. Cappelletti-Montano]{Beniamino Cappelletti-Montano}
 \address{Dipartimento di Matematica e Informatica, Universit\`a degli Studi di
 Cagliari, Via Ospedale 72, 09124 Cagliari, Italy}
 \email{b.cappellettimontano@gmail.com}

\author[A. De Nicola]{Antonio De Nicola}
 \address{CMUC, Department of Mathematics, University of Coimbra, 3001-501 Coimbra, Portugal}
 \email{antondenicola@gmail.com}

 \author[J.~C. Marrero]{Juan Carlos Marrero}
 \address{Unidad Asociada ULL-CSIC ``Geometr{\'\i}a Diferencial y Mec\'anica Geo\-m\'e\-tri\-ca''
Departamento de Matem\'aticas, Estad{\'\i}stica e Investigaci\'on Operativa, Facultad de Ciencias, Universidad de La Laguna, La Laguna, Tenerife, Spain}
 \email{jcmarrer@ull.edu.es}

\author[I. Yudin]{Ivan Yudin}
 \address{CMUC, Department of Mathematics, University of Coimbra, 3001-501 Coimbra, Portugal}
 \email{yudin@mat.uc.pt}



\thanks{This work was partially supported by CMUC -- UID/MAT/00324/2013, funded by the Portuguese
 Government through FCT/MEC and co-funded by the European Regional Development Fund through the Partnership Agreement PT2020 (A.D.N. and I.Y.),
 by MICINN (Spain) and European Union (Feder)  grants MTM2012-34478 and MTM 2015-64166-C2-2P (A.D.N. and J.C.M.), by the European Community IRSES-project GEOMECH-246981 (J.C.M.)
 and by Prin 2010/11 -- Variet\`{a} reali e complesse: geometria, topologia e analisi armonica -- Italy (B.C.M.), and by the exploratory research project in the frame of Programa Investigador FCT IF/00016/2013 (I.Y.). J.C.M. acknowledges the Centre for Mathematics of the University of Coimbra in Portugal for its support and hospitality in a visit where a part of this work was done.}

\begin{document}

\begin{abstract}
We establish a Hard Lefschetz Theorem for the de Rham cohomology of compact Vaisman manifolds. A similar result is proved for the basic cohomology with respect to the Lee vector field. Motivated by these results, we introduce the notions of a Lefschetz and of a basic Lefschetz locally conformal symplectic (l.c.s.)  manifold of the first kind.  We prove that the two notions  are equivalent if there exists a Riemannian metric such that the Lee vector field is unitary and parallel and its metric dual $1$-form coincides with the Lee $1$-form. Finally, we discuss several examples of compact l.c.s. manifolds of the first kind which do not admit compatible Vaisman metrics.
\end{abstract}


\maketitle

\section{Introduction}

\subsection{Antecedents and motivation}
\label{Antecedents-motivation}

It is well known that the global scalar pro\-duct on the space of $k$-forms in a
compact oriented Riemannian manifold $M$ of dimension $m$ induces an isomorphism
between the $k$th de Rham cohomology group $H^k(M)$ and the dual space of the
 $(m-k)$th de Rham cohomology group $H^{m-k}(M)$. So, using that the dimension
 of the de Rham cohomology groups is finite, we deduce the Poincar\'e-duality:
 the dimension of $H^k(M)$ is equal to the dimension of $H^{m-k}(M)$.

In addition, in some special cases, one may define a cano\-nical isomorphism
between the vector spaces $H^k(M)$ and $H^{m-k}(M)$. For instance, if $M$ is a
compact K\"ahler manifold of dimension $2n$ then, using the $(n-k)$th exterior
power of the symplectic $2$-form, one obtains an explicit isomorphism bet\-ween
$H^k(M)$ and $H^{2n-k}(M)$. This is known as the Hard Lefschetz isomorphism
for compact K\"ahler manifolds (see \cite{Ho}).

On the other hand, it is well known that the odd dimensional counterparts of
K\"ahler manifolds are Sasakian and co-K\"ahler manifolds (see \cite{Bl,BoGa}).
In these cases, one may also obtain a Hard Lefschetz isomorphism  as shown
in~\cite{CaNiYu} for Sasakian and in~\cite{tcm} for co-K\"ahler manifolds.
For a compact co-K\"ahler manifold the Hard
Lefschetz isomorphism depends only on the underlying cosymplectic structure and
for a Sasakian manifold it depends only on the corresponding contact structure.

A particular class of Hermitian manifolds which are related to K\"ahler, co-K\"ahler and Sasakian manifolds are Vaisman manifolds introduced
in \cite{Va0,Va1}. A Vaisman manifold is a locally conformal K\"ahler manifold $M$ which has non-zero parallel Lee $1$-form $\omega$.  This last condition is quite important in the complex case. However, we remark that in the compact quaternionic setting
it is not restrictive. In fact, for a non-hyperK\"ahler compact locally conformal hyperK\"ahler
manifold or for a non-quaternionic K\"ahler compact locally conformal quaternionic K\"ahler manifold one can always assume
that the Lee $1$-form is parallel (see the paper by Ornea and Piccinni \cite{OrPi}).

In this paper we will assume, without loss of generality, that the Lee $1$-form $\omega$ of a Vaisman manifold $M$ is unitary. If $J$ is the complex structure of $M$ then the $1$-form $\eta := -\omega \circ J$ is called the anti-Lee $1$-form of $M$. The \emph{Lee vector field} $U$ is defined as the metric dual of $\omega$, while the metric dual of $\eta$ is called the \emph{anti-Lee vector field} and is denoted by $V$.
The following properties of Vaisman manifolds can be found in
\cite{DrOr} and \cite{MaPa}:
\begin{itemize}
\item[-]
	the couple $(U, V)$ defines a flat foliation of rank $2$ on $M$
	which is transversely K\"ahler;
\item[-]
	the foliation on $M$ defined by $V$ is transversely co-K\"ahler;
\item[-]
	the orthogonal bundle to the foliation on $M$ defined by $U$ is integrable and the leaves of the corresponding foliation
are $c$-Sasakian manifolds.
\end{itemize}

The above results show that there is a close relationship between  Vaisman
manifolds on the one side and K\"ahler, co-K\"ahler, and Sasakian manifolds on the other side.
In fact, the mapping torus of a compact Sasakian manifold has a canonical Vaisman structure.
Moreover, in \cite{OrVe, OrVe2}, Ornea and Verbitsky proved that any compact
Vaisman manifold of dimension $2n+2$  is diffeomorphic to the mapping
torus of a compact Sasakian manifold $N$ of dimension $2n+1$.
However, note that this diffeomorphism does not preserve in general the geometric structure of the manifold.

In view of the above mentioned Hard Lefschetz theorems for compact Sasakian and co-K\"ahler manifolds,
a natural question arise: is there a Hard Lefschetz theorem for a compact Vaisman manifold?
The aim of this paper is to give a positive answer to this question.

\subsection{The results in the paper}\label{results}
The main result of this paper is the Hard Lefschetz theorem for a compact Vaisman manifold which may be formulated as follows.
\begin{theorem}\label{main-theorem}
Let $M$ be a compact Vaisman manifold of dimension $2n+2$ with Lee $1$-form
$\omega$, anti-Lee $1$-form $\eta$, Lee vector field $U$, and anti-Lee vector
field $V$.
 Then for each integer $k$ between $0$ and $n$,
 there exists an isomorphism
\[
\lef_k: H^{k}(M) \longrightarrow H^{2n+2-k}(M)
\]
 which can be computed by using the following properties:
\begin{enumerate}[$(i)$]
\item
for every $[\gamma] \in H^{k}(M)$, there is $\gamma' \in [\gamma]$ such that
\[
{\mathcal L}_{U}\gamma' = 0, \; \; i_V\gamma' = 0, \; \; L^{n-k+2}\gamma' = 0, \; \; L^{n-k+1}\epsilon_{\omega}\gamma' = 0;
\]
\item
if $\gamma' \in [\gamma]$ satisfies the conditions in $(i)$ then
\[
\lef_{k}[\gamma] = [\epsilon_{\eta}L^{n-k}(Li_{U}\gamma' - \epsilon_{\omega}\gamma')].
\]
\end{enumerate}
\end{theorem}
In this paper we write $\eps_\beta$ for the operator of the exterior multiplication by
a differential form $\beta$, and $L$ is defined to be $\eps_{d\eta}$. 

The map $\lef_k: H^k(M) \longrightarrow H^{2n+2-k}(M)$, $0 \leq k \leq n$, in
Theorem~\ref{main-theorem} will be called the \emph{Lefschetz isomorphism in degree
$k$} for the compact Vaisman manifold $M$.

In order to prove Theorem \ref{main-theorem}, we will first use some results which concern the basic cohomologies
of a compact oriented Riemannian manifold with respect to some Riemannian foliations.  In fact, these results on basic cohomologies allow us to prove the $U$-basic Hard Lefschetz theorem below for a compact Vaisman manifold (this theorem is essential in the proof of Theorem  \ref{main-theorem}).

\begin{theorem}\label{basic-theorem}
Let $M$ be a compact Vaisman manifold of dimension $2n+2$ with Lee $1$-form
$\omega$, anti-Lee $1$-form $\eta$, Lee vector field $U$, and anti-Lee vector
field $V$.
Denote by $H_B^*(M,U)$  the basic cohomology of $M$ with respect to  $U$. Then for each
integer $k$ between $0$ and $n$,
there exists an isomorphism
\[
\lef_k^{U}: H^{k}_{B}(M,U) \longrightarrow H^{2n+1-k}_{B}(M,U)
\]
 which can be computed by using the following properties:
\begin{enumerate}[$(i)$]
\item
for every $[\beta]_{U} \in H^{k}_{B}(M,U)$, there is $\beta' \in [\beta]_{U}$ such that
\begin{equation}\label{basic-conditions}
 i_V\beta' = 0, \; \; L^{n-k+1}\beta' = 0;
\end{equation}
\item
if $\beta' \in [\beta]_{U}$ satisfies the conditions in $(i)$ then
\[
\lef_{k}^{U}[\beta]_{U} = [\epsilon_{\eta}L^{n-k}\beta']_{U}.
\]
\end{enumerate}
\end{theorem}

The map $ \lef_k^U: H_B^k(M,U) \longrightarrow H^{2n+1-k}_B(M,U)$, for $0 \leq k \leq n$, will be called the \emph{$U$-basic Lefschetz isomorphism in degree $k$} associated with the compact Vaisman manifold~$M$.


\begin{remark}
As we mentioned in Section \ref{Antecedents-motivation}, a compact Vaisman manifold $M$ of dimension $2n+2$ is diffeomorphic to the mapping torus of a compact Sasakian manifold $S$ of dimension $2n+1$ (see \cite{OrVe2}).
Note that in general the original Vaisman structure  differs from the one obtained from the Sasakian structure by the mapping torus construction. One can obtain  a Hard Lefschetz isomorphism for the de Rham cohomology of $M$ using the results in \cite{CaNiYu} for this new Vaisman structure. However, this isomorphism is not given in terms of the original Vaisman structure on $M$ (as in our Theorem \ref{main-theorem}).
For this reason, our approach in this paper is to use some results on the basic cohomology associated with the Lee vector field $U$ of $M$.
\end{remark}

For a Vaisman manifold $M$ of dimension $2n+2$,
the couple $(\omega,\eta)$ of the Lee and anti-Lee $1$-forms
defines a locally conformal symplectic (l.c.s.)  structure of the first kind
with anti-Lee vector field $V$ and infinitesimal automorphism $U$ (see \cite{Va}
and Section \ref{Compact-Sasakian-Vaisman} for the definition of an l.c.s. structure of the first kind). Note that the definition of the Hard Lefschetz and the basic Hard Lefschetz isomorphism associated with $M$ only depends on the l.c.s. structure of the first kind.
Hence, both isomorphisms provide obstructions for an l.c.s.  manifold of the first kind to admit Vaisman structures.

Now, let $M$ be a compact manifold of dimension $2n+2$ endowed with an l.c.s. structure of the first kind $(\omega, \eta)$. Suppose that $U$ and $V$ are the anti-Lee and Lee vector field, respectively, on $M$.

Then, the previous results suggest us to introduce the following \emph{Lefschetz relation}  between the cohomology groups $H^k(M)$ and $H^{2n+2-k}(M)$ , for $0\leq k \leq n$,

\begin{multline*}
	R_{\lef_k} =
	\left\{ ([\gamma], [\epsilon_{\eta}L^{n-k}(Li_{U}\gamma -
\epsilon_{\omega}\gamma)]) \right|
\gamma \in \Omega^k(M), \; d\gamma = 0, \; {\mathcal L}_{U}\gamma = 0, \;
i_V\gamma = 0,\\
\left. L^{n-k+2}\gamma = 0, \; L^{n-k+1}\epsilon_{\omega}\gamma = 0\right\}.
\end{multline*}

Similarly, we define the \emph{$U$-basic Lefschetz relation} between the basic
cohomology groups $H^k_B(M, U)$ and $H^{2n+1-k}_B(M, U)$, for $0\le k\le n$, by
\begin{equation*}
	R_{\lef_k}^B = \left\{ { ([\beta]_{U}, [\epsilon_{\eta}L^{n-k}\beta]_{U})} \;
	\middle| \; \beta
	\in \Omega^k_{B}(M, U), \; d\beta = 0, \; i_V\beta = 0, \;  L^{n-k+1}\beta
	= 0\right\}.
\end{equation*}
	An l.c.s. structure on $M$ of the first kind is said to be:
\begin{itemize}
\item[-]
	\emph{Lefschetz} if, for every $0\leq k \leq n$, the relation
	$R_{\lef_k}$ is the graph of an isomorphism $\lef_k: H^k(M) \longrightarrow
	H^{2n+2-k}(M)$;
 \item[-]
	 \emph{Basic Lefschetz} if, for every $0\leq k \leq n$, the relation $R_{\lef_k}^B$ is the graph of an isomorphism $\lef_k^U: H^k_B(M, U) \longrightarrow H^{2n+1-k}_{B}(M, U)$.
\end{itemize}
It is not clear what is the relation between the Lefschetz property and the basic
Lefschetz property in general.
However, we may prove the following result.
\begin{theorem}\label{basic-Hard-Lef-property}
Let $M$ be a compact manifold of dimension $2n+2$ endowed with an l.c.s. structure of the first kind $(\omega, \eta)$ such that the Lee vector field $U$ is unitary and parallel with respect to a Riemannian metric $g$ on $M$ and
\[
\omega(X) = g(X, U), \; \; \mbox{ for } X \in {\mathfrak X}(M).
\]
Then:
\begin{enumerate}[$(1)$]
\item
The structure $(\omega, \eta)$ is Lefschetz if and only if it is $U$-basic Lefschetz.
\item
If the structure $(\omega, \eta)$ is Lefschetz (or, equivalently, $U$-basic Lefschetz),
then for each $1\leq k \leq n$ there exists a non-degenerate bilinear form
\begin{gather*}
\psi: H_{B}^{k}(M, U) \times H_{B}^{k}(M, U) \longrightarrow \R \\
\psi([\beta]_{B}, [\beta']_{B})= \int_{M} [\omega]\cup\lef_{k}^U [\beta]\cup[\beta']
\end{gather*}
which is skew-symmetric for odd $k$ and symmetric for even $k$. As a consequence,
\begin{equation}\label{Betti-numbers}
b_k(M) - b_{k-1}(M) \mbox{ is even if } k \mbox{ is odd and } 1 \leq k \leq n,
\end{equation}
where $b_r(M)$ is the $r$\mbox{\rm th} Betti number of $M$.
\end{enumerate}
\end{theorem} We remark that relations in (\ref{Betti-numbers}) are well-known properties of the Betti numbers of a compact Vaisman manifold of dimension $2n+2$ (see \cite{DrOr,Va1}).

\subsection{Organization of the paper}
In Section~\ref{Compact-Sasakian-Vaisman} we review some known results on locally  conformal symplectic, contact and Vaisman manifolds.  In Section~\ref{proof12}, we will discuss several results on the basic cohomology of some transversely oriented Riemannian foliations on compact Riemannian manifolds. In Sections~\ref{basic-proof} and \ref{proof14}, we will prove Theorems~\ref{basic-theorem}
and~\ref{basic-Hard-Lef-property}, respectively. As a consequence Theorem~\ref{main-theorem}
is also proved in Section~\ref{proof14}. Finally, in Section \ref{sec:examples}, we give several examples of compact l.c.s. manifolds of the fist kind which do not admit compatible Vaisman metrics. Some of these examples satisfy the Lefschetz property and the basic Lefschetz property and others not.

\bigskip

All manifolds considered in this paper will be assumed to be smooth and connected.
For wedge product, exterior derivative and interior product we use the conventions as in Goldberg's book \cite{Go}.

\section{ Locally conformal symplectic, contact and Vaisman manifolds}\label{Compact-Sasakian-Vaisman}
In this section we will review the definition of a locally conformal symplectic (l.c.s.) structure of the first kind,  of a contact structure and of a Vaisman manifold. More details can be found in  \cite{BaMa,DrOr, Ornea05, Va}.

An \emph{l.c.s. structure of the first kind} on a manifold $M$ of dimension $2n+2$ is a couple $(\omega, \eta)$ of $1$-forms such that:
\begin{enumerate}[$(i)$]
\item
$\omega$ is closed;
\item
the rank of $d\eta$ is $2n$ and $\omega \wedge \eta \wedge (d\eta)^{n}$ is a volume form.
\end{enumerate}
The form $\omega$ is called the \emph{Lee $1$-form} while $\eta$ is said to be the \emph{anti-Lee} \emph{$1$-form}.

If $(\omega, \eta)$ is an l.c.s. structure of the first kind on $M$ then
there exists a unique vector field $V$, \emph{the anti-Lee vector field} of $M$, which is characterized by the following conditions
\[
\omega(V) = 0, \; \; \eta(V) = 1, \; \; i_Vd\eta = 0.
\]
Moreover,
there exists a unique vector field $U$, \emph{the Lee vector field} of $M$, which is characterized by the following conditions
\[
\omega(U) = 1, \; \; \eta(U) = 0, \; \; i_Ud\eta = 0.
\]
\begin{remark}
If $(\omega, \eta)$ is an l.c.s. structure of the first kind then
the $2$-form
\[
\Omega := d\eta + \eta \wedge \omega
\]
is non-degenerate and
\begin{equation}\label{lck}
d\Omega = \omega \wedge \Omega.
\end{equation}
Moreover, the Lee vector field $U$ satisfies the condition
\(
{\mathcal L}_{U}\Omega = 0.
\)
In other words, $\Omega$ is an l.c.s. structure  of the first kind in the sense of Vaisman \cite{Va} with Lee $1$-form $\omega$ and infinitesimal automorphism $U$.
Conversely, if $\Omega$ is an l.c.s. structure of the first kind with Lee $1$-form $\omega$ and infinitesimal automorphism $U$ then the rank of $d\eta$ is $2n$ and $\omega \wedge \eta \wedge (d\eta)^n$ is a volume form, with $\eta$ the $1$-form on $M$ given by
\[
\eta = -i_U\Omega.
\]
\end{remark}

The typical example of a compact l.c.s. manifold of the first kind is the product of a compact contact manifold with the circle $S^1$. We recall that a $1$-form $\eta$ on a manifold $N$ of dimension $2n+1$ is said to be a contact structure if $\eta \wedge (d\eta)^n$ is a volume form. In such a case, there exists a unique vector field $\xi$ on $N$,   called the Reeb vector field, which is characterized by the following conditions
\[
i_{\xi}\eta = 1, \; \; \; i_{\xi}d\eta = 0.
\]

If $\eta$ is a contact structure on $N$ then the product manifold $M = N \times S^1$ admits a l.c.s. structure of the first kind which is given by $(pr_2^*\theta, pr_1^*\eta)$, where $\theta$ is the volume form of length 1 on $S^1$ and $pr_1: M \to N$ and $pr_2: M \to S^1$ are the canonical projections. The Lee and anti-Lee vector fields on $M$ are the Reeb vector field $\xi$ and $E$, respectively, with $E$ the canonical vector field on $S^{1}$.

A \emph{Vaisman manifold} is an l.c.s. manifold of the first kind $(M, \omega, \eta)$ which carries a Riemannian metric $g$ such that:
\begin{enumerate}
\item
the tensor field $J$ of type $(1, 1)$, given by
\begin{equation}\label{compatibility}
g(X, JY) = \Omega(X, Y), \; \; \mbox {for } X, Y \in {\mathfrak X}(M),
\end{equation}
is a complex structure which is compatible with $g$, that is,
\begin{equation}\label{hermitian}
g(JX, JY) = g(X, Y);
\end{equation}
\item
the Lee $1$-form $\omega$ is parallel with respect to $g$.
\end{enumerate}
\begin{remark}\label{Lee-unitary}
If $\omega$ is the Lee $1$-form of a Vaisman manifold $M$ then $\|\omega\|$ is a  positive constant.
We will assume, without  loss of  generality, that $\omega$ is unitary which implies that
the anti-Lee and Lee vector fields $U$ and $V$  are  also unitary.
\end{remark}
If $(M, J, g)$ is a Vaisman manifold one may prove that $U$ is parallel (and, thus, Killing), $V$ is Killing and
\[
[U, V] = 0, \; \; {\mathcal L}_{U}J = 0, \; \; {\mathcal L}_VJ = 0.
\]
 So, the foliation of rank $2$ which is generated by $U$ and $V$ is transversely K\"ahler.

 As we expected, the existence of a Vaisman structure on a compact manifold implies several topological restrictions. In fact, from (\ref{lck}), (\ref{compatibility}) and (\ref{hermitian}), it follows that a Vaisman manifold is locally conformal K\"ahler (l.c.K.) and the existence of a l.c.K. structure on a compact manifold $M$ also implies that the topology of $M$ must satisfy some conditions (see \cite{OrneaVerbitsky11}).

\section{Some results on the basic cohomology of Riemannian foliations}
\label{proof12}

Let $\fol$ be a transversely oriented Riemannian foliation  on a compact Riemannian manifold $P$ with a bundle-like metric $g$.
We define the basic de Rham complex with respect to $\fol$ by
\[
\Omega^{k}_{B}(P,\fol):=\left\{\alpha \in \Omega^{k}(P) \;\middle|\; i_X\alpha=0,\; \lie_X\alpha=0,  \; \forall X\in\Gamma (T\fol)  \right\}.
\]
 The exterior derivative $d$ restricts to a cohomology operator $d_\fol$ on basic forms. The corresponding basic cohomology will be denoted by $H^{k}_{B}(P,\fol)$.

Recall that a vector field $W$ is said to be \emph{foliated} with respect to the foliation $\fol$ if $[W,\Gamma (T\fol)]\subset \Gamma (T\fol)$.
Let $(P,\fol)$ be a foliated manifold and $W$ a parallel unitary foliated vector field on $P$ which is  orthogonal to $\fol$.
Then, we can enlarge the foliation $\fol$ by adding the vector field $W$, defining a new foliation
$\folw:=\langle\fol,  W \rangle$. 
The following theorem relates the corresponding basic cohomologies $H^*_{B}(P, \fol)$ and $H^*_{B}(P,\folw)$.


\begin{theorem}
\label{basic-cohomologies}
Let $\fol$ be a transversely oriented Riemannian foliation on a compact oriented Riemannian manifold $(P,g)$ of dimension $p$
 and $W$ be a unitary and parallel foliated vector field on $P$ that is orthogonal to $\fol$.
 Denote by  $w$ the metric dual $1$-form of $W$ and let $\folw:=\langle\fol,  W \rangle$.
Then for each integer $k$ between $0$ and $p$, the map
\[
 [(Id, \epsilon_{w})]:  H^k_{B}(P,\folw) \oplus H^{k-1}_{B}(P,\folw) \longrightarrow H^{k}_{B}(P,\fol)
\]
defined by
\begin{equation}\label{Id-epsilon_w-k}
[(Id, \epsilon_{w})]([\beta]_{\folw}, [\beta']_{\folw}) = [\beta + w \wedge \beta']_{\fol}
\end{equation}
is an isomorphism.
\end{theorem}

Considering as a special case $\fol$ to be the zero foliation on $P$, one can relate the de~Rham cohomology $H^{*}(P)$
to the basic cohomology $H^*_{B}(P,W)$ with respect to the vector field $W$.

\begin{corollary}
\label{basic-cohomology-DeRham-cohomology}
Let $W$ be a unitary and parallel vector field on a compact oriented Riemannian manifold $(P,g)$ of dimension $p$.
Denote by  $w$ the metric dual $1$-form of $W$.
Then for each integer $k$ between $0$ and $p$, the map
\[
 [(Id, \epsilon_{w})]:  H^k_{B}(P,W) \oplus H^{k-1}_{B}(P,W) \longrightarrow H^{k}(P)
\]
defined by
\begin{equation}\label{Id-epsilon_w-k2}
[(Id, \epsilon_{w})]([\beta]_{W}, [\beta']_{W}) = [\beta + w \wedge \beta']
\end{equation}
is an isomorphism.
\end{corollary}

The rest of this section is devoted to prove Theorem \ref{basic-cohomologies}.
First we need to recall some preliminary results.

Denote by $\star$ the Hodge star isomorphism on $P$ and by $\delta$ the codifferential, defined for a $k$-form $\theta$ on $P$ by
\[
\delta \theta = (-1)^{pk + p +1} \star d \;\star \theta.
\]
Note that (see page 97 in \cite{Go})
\begin{equation}\label{star-star}
\star \; \star = (-1)^{k(p-k)} Id, \qquad \star \, \delta = (-1)^p d \star.
\end{equation}
Let $\{W_1, \dots, W_p\}$ be a local orthonormal basis of vector fields on $P$ and $\{w^1, \dots, w^p\}$ the corresponding dual basis of $1$-forms.
We have that
\begin{equation}
\label{diff-co-diff}
d \theta = \displaystyle \sum_{j=1}^{p} \epsilon_{w^j}\nabla_{W_{j}}\theta, \qquad \delta \theta = - \sum_{j=1}^{p} i_{W_{j}}\nabla_{W_{j}}\theta,
\end{equation}
where $\nabla$ is the Levi-Civita connection on $P$ and $\eps_{w^j}$ is the operator of the exterior multiplication by the $1$-form $w^j$.

Using (\ref{diff-co-diff}), the following result clearly holds.
\begin{lemma}
\label{parallel-closed-coclosed}
If $\theta$ is a parallel $k$-form on $P$, then $\theta$ is harmonic.
\end{lemma}

Now, one may define
the basic codifferential $\delta_\fol:\Omega^{k}_{B}(P,\fol)\to\Omega^{k-1}_{B}(P,\fol)$ as $\delta_\fol=\Pi_\fol\circ\delta$, where $\Pi_\fol: \Omega^k(P)\to \Omega^k_B(P, \fol)$ is the orthogonal projection on the space of basic forms. Then,  basic Hodge theory for the basic Laplacian $\Delta_\fol= d_\fol \delta_\fol +\delta_\fol d_\fol$ holds.  In fact, if $\Omega^k_{\Delta_\fol}(P)$ is the space of basic-harmonic $k$-forms, we have the orthogonal decomposition
\begin{align*}
\Omega^{k}_{B}(P,\fol)= \Delta_\fol(\Omega^{k}_{B}(P,\fol)) \oplus \Omega^k_{\Delta_\fol}(P) =
\Imm d_\fol \oplus \Imm \delta_\fol \oplus \Omega^k_{\Delta_\fol}(P),
\end{align*}
(see \cite{PaRi}). Thus, there exists a basic Green operator $G_\fol:\Omega^k_B(P, \fol)\to \Omega^k_{\Delta_\fol}(P)^{\perp} = \Delta_\fol(\Omega^k_B(P, \fol))$ which is characterized by the condition
\begin{equation}\label{eq:basichodge2}
\Delta_\fol G_\fol =\Id-\harm_\fol,
\end{equation}
where $\harm_\fol: \Omega^k_B(P, \fol) \to \Omega^k_{\Delta_\fol}(P)$ is the orthogonal projection. It is clear that $G_\fol\harm_\fol=\harm_\fol G_\fol=0$ and
\begin{equation}\label{eq:basichodge1}
G_\fol \Delta_\fol =\Id-\harm_\fol,
\end{equation}
(for more details, see \cite{PaRi}).

 Now, we prove that the following property of the standard Green operator also holds in the foliated case.
\begin{lemma}\label{basic-warner}
The basic Green operator $G_\fol$ commutes with any linear operator on $\Omega^{*}_{B}(P,\fol)$ which commutes with the basic Laplacian $\Delta_\fol$.
\end{lemma}
\begin{proof}
Let $S$ be a linear operator on $\Omega^{*}_{B}(P,\fol)$ such that
\begin{equation}\label{eq:commutelaplacian}
\Delta_\fol S = S\Delta_\fol.
\end{equation}
First, we prove that $\harm_\fol S=S\harm_\fol$.
By composing \eqref{eq:basichodge1} with $S$ on the left we have
\begin{equation}\label{eq:basichodge-left}
S G_\fol\Delta_\fol  +S\harm_\fol = S .
\end{equation}
On the other hand, by composing with $S$ on the right and using \eqref{eq:commutelaplacian} we get
\begin{equation}\label{eq:basichodge-right}
 G_\fol S \Delta_\fol+\harm_\fol S = S.
\end{equation}
From \eqref{eq:basichodge-left} and \eqref{eq:basichodge-right} we obtain
\begin{equation*}
G_\fol S \Delta_\fol+\harm_\fol S = S G_\fol\Delta_\fol  +S\harm_\fol.
\end{equation*}
Now, by composing with $\harm_\fol$ on the right and using that $\Delta_\fol \harm_\fol=0$ we get
\begin{equation}\label{eq:HTH=TH}
\harm_\fol S\harm_\fol = S \harm_{\fol}^2 = S \harm_\fol .
\end{equation}
%
Moreover, starting from \eqref{eq:basichodge2} and performing steps similar to the above, one gets
\begin{equation}\label{eq:HTH=HT}
\harm_\fol S\harm_\fol = \harm_{\fol}^2  S=\harm_\fol S.
\end{equation}
From \eqref{eq:HTH=TH} and \eqref{eq:HTH=HT} we conclude that $\harm_\fol$ commutes with $S$ as claimed.

Now, from \eqref{eq:basichodge1} we have
\begin{equation*}
G_\fol \Delta_\fol S G_\fol = (\Id-\harm_\fol) S G_\fol= S G_\fol - \harm_\fol S G_\fol
\end{equation*}
Since $\harm_\fol$ commutes with $S$ and $\harm_\fol G_\fol=0$ we get
\begin{equation}\label{eq:SG}
G_\fol \Delta_\fol S G_\fol =  S G_\fol .
\end{equation}
Moreover, as $S$ commutes with $\Delta_\fol$ and using \eqref{eq:basichodge2} we have
\begin{equation*}
G_\fol \Delta_\fol S G_\fol = G_\fol S \Delta_\fol  G_\fol 
= G_\fol S -  G_\fol S \harm_\fol.
\end{equation*}
Since $S$ commutes with $\harm_\fol$ and $G_\fol  \harm_\fol = 0$ we obtain
\begin{equation}\label{eq:GS}
G_\fol \Delta_\fol S G_\fol =   G_\fol S.
\end{equation}
From \eqref{eq:SG} and \eqref{eq:GS} we  conclude that $S G_\fol =   G_\fol S$.
\end{proof}

 Next, we will prove Theorem \ref{basic-cohomologies}.

\begin{proof}[Proof of Theorem \ref{basic-cohomologies}]
We will proceed in two steps.

\medskip

\noindent {\bf First step.}
Define
\[
\Omega^*_B(P, \fol)^{\lie_W}:=\left\{\beta \in \Omega^{k}_B(P, \fol) \;\middle|\;  { \lie_W \beta=0}  \right\}.
\]
We will prove that the inclusion map
\begin{equation*}
		i: \Omega^*_B(P, \fol)^{\lie_W}\longhookrightarrow \Omega^*_B(P, \fol)
	\end{equation*}
induces an isomorphism between the corresponding cohomology groups
\begin{equation*}
		[i]:H^*_B(P, \fol)^{\lie_W} \longhookrightarrow H^*_B(P, \fol).
	\end{equation*}
 First of all, we will see that $[i]$ is injective.

Let $\beta\in \Omega^k_B(P, \fol)^{\lie_W}$ be a closed form such that $\beta$ is exact in $\Omega^*_B(P, \fol)$.
Let $\gamma:=\delta_\fol G_\fol \beta$, where $\delta_\fol$ is the basic codifferential and $G_\fol$ is the basic Green operator with respect to $\fol$.
Then, by using \eqref{eq:basichodge2} we have
\begin{align*}
d\gamma &=d \delta_\fol G_\fol \beta= \Delta_\fol G_\fol \beta- \delta_\fol d G_\fol \beta\\
 &=(\Id-\harm_\fol)\beta-\delta_\fol G_\fol d\beta=\beta,
\end{align*}
since  $G_\fol$ commutes with $d$,  $d\beta = 0$ and $\harm_\fol \beta = 0$. 

It is clear that $\gamma$ is basic. Thus, in order to show that $\beta$ is exact in $\Omega^*_B(P, \fol)^{\lie_W}$ we have to check that $\lie_W\gamma=0$.

To get $\lie_W\gamma=0$, it is enough to have  $[\lie_W,\delta_\fol]=0$. In fact, if $[\lie_W,\delta_\fol]=0$, then $[\lie_W,d]=0$ implies $[\lie_W,\Delta_\fol]=0$, and hence by Lemma~\ref{basic-warner} $[\lie_W,G_\fol]=0$. Thence
\begin{align*}
\lie_W\gamma=\lie_W \delta_\fol G_\fol \beta= \delta_\fol G_\fol  \lie_W \beta=0.
\end{align*}
It is left to prove that $[\lie_W,\delta_\fol]=0$. Since $W$ is parallel, by \cite[page 109]{Go} we have $[\lie_W,\delta]=0$.
But $\delta_\fol=\Pi_\fol\circ\delta$.
We will now show that $[\lie_W,\Pi_\fol]=0$.  

First, note that since $W$ is a foliated vector field we have
\begin{align}\label{eq:foliated-invariant}
\lie_W \Omega^k_B(P, \fol)\subset \Omega^k_B(P, \fol).
\end{align}
 Indeed, for each $X\in\Gamma (T\fol)$ and $\alpha\in\Omega^k_B(P, \fol)$ we get
\begin{equation*}
i_X\lie_W\alpha=\lie_W i_X\alpha + i_{[X,W]}\alpha=0
\end{equation*}
and
\begin{equation*}
\lie_X\lie_W\alpha=\lie_W \lie_X\alpha + \lie_{[X,W]}\alpha=0.
\end{equation*}
Now, note that \eqref{eq:foliated-invariant} implies
\begin{align*}
\lie_W^* \Omega^k_B(P, \fol)^\perp\subset \Omega^k_B(P, \fol)^\perp,
\end{align*}
where $\lie_W^*$ is the adjoint operator of $\lie_W$ with respect to the global scalar product in $\Omega^k(P)$. But $\lie_W^*=-\lie_W$ since $W$ is Killing \cite[page 109]{Go} and hence we have
\begin{align}\label{eq:foliated-ortogonal-invariant}
\lie_W \Omega^k_B(P, \fol)^\perp\subset \Omega^k_B(P, \fol)^\perp.
\end{align}
Let $\theta\in\Omega^k(P)$. We can decompose it as
\[
\theta=\theta_\fol+\theta_{\fol^\perp}
\]
with  $\theta_\fol=\Pi_\fol\theta\in\Omega^k_B(P, \fol)$ and $\theta_{\fol^\perp}\in\Omega^k_B(P, \fol)^\perp$. Thus
\[
\lie_W\theta=\lie_W\theta_\fol + \lie_W\theta_{\fol^\perp}.
\]
Using \eqref{eq:foliated-invariant} and \eqref{eq:foliated-ortogonal-invariant} we obtain that $\lie_W\theta_{\fol}\in\Omega^k_B(P, \fol)$ and $\lie_W\theta_{\fol^\perp}\in\Omega^k_B(P, \fol)^\perp$. Hence
\[
\Pi_\fol\lie_W\theta=\lie_W\theta_\fol=\lie_W\Pi_\fol\theta,
\]
for every $\theta\in\Omega^k(P)$. Thus, $[\lie_W,\Pi_\fol]=0$ as claimed.

It follows that
\[
[\lie_W,\delta_\fol]=[\lie_W,\Pi_\fol]\delta+ \Pi_\fol[\lie_W,\delta]=0.
\]
Thus, we can conclude that $\lie_W\gamma=0$. Therefore $\beta$ is exact also in $\Omega^*_B(P, \fol)^{\lie_W}$.  This proves the injectivity of $[i]$.

Next, we will see that $[i]$ is surjective.

Now, let $\beta\in\Omega^k_{\Delta_\fol}(P)$. Then $\lie_W \beta=d i_W \beta$ is exact.
 Moreover, from (\ref{eq:foliated-invariant}), $\lie_W\beta \in \Omega^k_B(P, \fol)$. Hence, by basic Hodge theory we have
\begin{align*}
\lie_W\beta &=d \delta_\fol G_\fol \lie_W\beta\\
&=d  G_\fol \delta_\fol \lie_W\beta\\
&=d  G_\fol  \lie_W\delta_\fol\beta=0,
\end{align*}
where we used Lemma~\ref{basic-warner} and $[\lie_W,\delta_\fol]=0$.
This shows that $\beta$ is $\lie_W$-invariant and hence belongs to the image of $[i]$, proving surjectivity.

\medskip

\noindent {\bf Second step.}
For a closed $\fol$-basic $k$-form  $\beta\in \Omega^k_{B}(P,\fol)$,
let us denote by $[\beta]_{\fol}$ its cohomology class in $H^k_{B}(P,\fol)$.
We will now prove that
for each integer $k$ between $0$ and $p$, the map
\begin{equation}\label{iso-well}
 ([i_W\eps_w ]_{\folw}, [i_W]_{\folw}):  H^{k}_{B}(P,\fol)^{\lie_W} \longrightarrow  H^k_{B}(P,\folw) \oplus H^{k-1}_{B}(P,\folw)
\end{equation}
defined by
\begin{equation*}\label{Idwk}
([i_W\eps_w ]_{\folw}, [i_W]_{\folw})([\beta]_{\fol}) =  ([i_W\eps_w \beta]_{\folw}, [i_W\beta]_{\folw})
\end{equation*}
is an isomorphism.
Let us start by checking that it is well defined.
Consider $\beta\in \Omega^{k}_{B}(P,\fol)^{\lie_W}$.
We will first show that
\begin{equation}\label{ieps-well}
i_W\eps_w \beta\in\Omega^{k}_{B}(P,\folw).
\end{equation}


Now, for any $X\in\Gamma (T\fol)$  we have that the graded commutator $[i_X, \eps_w]$ satisfies
\begin{equation}\label{ieps}
[i_X, \eps_w]=\eps_{i_X w}=0,
\end{equation}
as $i_X w=g(W,X)=0$, and
\begin{equation}\label{leps}
[\lie_X, \eps_w]=\eps_{\lie_X w}=0,
\end{equation}
as  $\lie_X w=i_X dw=0$.
Thus, by using \eqref{ieps} and \eqref{leps} we get
\begin{align*}
\lie_X(i_W\eps_w \beta)&=i_W\lie_X\eps_w \beta + i_{[X,W]}\eps_w \beta\\
                &=i_W\eps_w \lie_X\beta - \eps_w i_{[X,W]} \beta=0,
\end{align*}
since  $[X,W]\in\Gamma (T\fol)$ and $i_{[X,W]} w=0$. Moreover,
\begin{equation*}
i_X (i_W\eps_w \beta)=i_W\eps_w i_X \beta=0.
\end{equation*}
Further, we have that $i_W(i_W\eps_w \beta)=0$ and
\begin{align*}
\lie_W(i_W\eps_w \beta)&=i_W\lie_W\eps_w \beta =i_W\eps_w \lie_W\beta + i_{W} \eps_{\lie_Ww} \beta=0,
\end{align*}
as $\lie_W\beta=0$,  $[\lie_W,\eps_w]=\eps_{\lie_W w}$  and $\lie_Ww=i_Wdw+d1=0$. This completes the proof of \eqref{ieps-well}.

Now, we will check that
\begin{equation}\label{iW-well}
i_W \beta\in\Omega^{k-1}_{B}(P,\folw).
\end{equation}
Indeed, the conditions $i_W (i_W \beta)=0$ and $\lie_W (i_W \beta)=0$ are clearly satisfied.
Moreover, for any $X\in\Gamma (T\fol)$ we have $i_X (i_W \beta)=-i_W (i_X \beta)=0$ and
\begin{equation*}
\lie_X (i_W \beta)=i_W \lie_X \beta+ i_{[X,W]}\beta=0.
\end{equation*}
Moreover, note that $i_W\beta$ and $i_W\eps_w \beta$  are closed, assuming that $\beta\in \Omega^{k}_{B}(P,\fol)^{\lie_W}$ is closed. Indeed, we have
\begin{equation*}
d (i_W \beta)=\lie_W \beta -i_W d \beta=0
\end{equation*}
and
\begin{align*}
d (i_W\eps_w \beta)&=\lie_W \eps_w \beta -i_W d \eps_w \beta\\
                    &=\eps_w \lie_W \beta  +i_W  \eps_w d\beta=0,
\end{align*}
since $\lie_W$ commutes with $\eps_w$ and  $dw=0$.   

Suppose $\beta,\beta'\in \Omega^{k}_{B}(P,\fol)^{\lie_W}$ are  closed forms, with $\beta-\beta'=d\gamma$, for some $\gamma\in \Omega^{k-1}_{B}(P,\fol)^{\lie_W}$. Then
\begin{align*}
i_W\eps_w \beta- i_W\eps_w \beta' &=i_W\eps_w d\gamma \\
                                    &= - i_W d \eps_w \gamma\\
                                     & = d i_W \eps_w \gamma -\lie_W \eps_w \gamma\\
                                      & = d i_W \eps_w \gamma -\eps_w \lie_W  \gamma= d (i_W \eps_w \gamma)
\end{align*}
and  $i_W \eps_w \gamma\in\Omega^{k-1}_{B}(P,\folw)$, due to \eqref{ieps-well}.
Therefore
\begin{equation*}
[i_W\eps_w \beta]_\folw = [i_W\eps_w \beta']_\folw.
\end{equation*}
We also have that
\begin{equation*}
[i_W \beta]_\folw = [i_W \beta']_\folw.
\end{equation*}
Indeed,
\begin{equation*}
i_W \beta-i_W \beta'= i_W d\gamma= -d(i_W \gamma)
\end{equation*}
and $i_W \gamma$ belongs to $\Omega^{k-2}_{B}(P,\folw)$, due to \eqref{iW-well}.
We conclude that the map \eqref{iso-well} is well defined.
Consider now the map
\begin{equation}\label{inv-well}
 [ (Id,\eps_w)]:     H^k_{B}(P,\folw) \oplus H^{k-1}_{B}(P,\folw) \longrightarrow H^{k}_{B}(P,\fol)^{\lie_W}
\end{equation}
defined by
\begin{equation*}
[ (Id,\eps_w)]([\alpha]_{\folw},[\beta]_{\folw}) =  [\alpha +\eps_w\beta]_{\fol}.
\end{equation*}
Let us check that it is well defined.
Clearly, we have  $\alpha\in\Omega^{k}_{B}(P,\folw)\subset\Omega^{k}_{B}(P,\fol)^{\lie_W}$.
Moreover, let $\beta\in\Omega^{k-1}_{B}(P,\folw)$. Then
$\lie_W(\eps_w\beta)=\eps_w\lie_W\beta=0$.  
Now, let $X\in\Gamma (T\fol)$. Then
\[
i_X(\eps_w\beta)=-\eps_w i_X\beta=0,
\]
due to \eqref{ieps}, and
\[
\lie_X(\eps_w\beta)=\eps_w \lie_X\beta=0,
\]
thanks to \eqref{leps}.
We conclude that
\[
\alpha +\eps_w\beta\in\Omega^{k}_{B}(P,\fol)^{\lie_W}.
\]
Suppose $d\alpha=0$ and $d\beta=0$. Then
\[
d(\alpha +\eps_w\beta)=d\alpha -\eps_w d\beta=0.
\]
Now, assume that $\alpha,\alpha'\in\Omega^{k}_{B}(P,\folw)$ with $\alpha-\alpha'=d\gamma$, for some $\gamma\in\Omega^{k-1}_{B}(P,\folw)$.
Then $\gamma\in\Omega^{k-1}_{B}(P,\fol)^{\lie_W}$ and $[\alpha]_\fol=[\alpha']_\fol$.
Suppose $\beta,\beta'\in\Omega^{k-1}_{B}(P,\folw)$, with $\beta-\beta'=d\tilde\gamma$, for some $\tilde\gamma\in\Omega^{k-2}_{B}(P,\folw)$.
Then,  we have $[\eps_w\beta]_\fol=[\eps_w\beta']_\fol$.
Indeed
\[
d \eps_w \tilde\gamma= \eps_w\beta -\eps_w\beta'
\]
and it is easy to check that $\eps_w \tilde\gamma\in\Omega^{k-1}_{B}(P,\fol)^{\lie_W}$.

This completes the proof that the map \eqref{inv-well} is well defined.
Finally, one can easily check that it is the inverse of the map \eqref{iso-well}.
\end{proof}


\section{Proof of Theorem~\ref{basic-theorem}}
\label{basic-proof}
\begin{proof}
 We will proceed in two steps.

 \noindent {\bf First step.} Using the results in \cite{elkacimi} on Hard Lefschetz isomorphisms for transversely K\"ahler foliations and the fact that the foliation generated by the Lee and anti-Lee vector fields on $M$ is transversely K\"ahler, we will see that it is possible to define a morphism
\[
T_k: H^{2n+1-k}_B(M, U) \to H^{k}_B(M, U), \; \; \; \mbox{ for } 0 \leq k \leq n.
\]
This map will be just the inverse morphism of the $U$-basic Lefschetz isomorphism in degree $k$.

Since the anti-Lee vector field $V$ of our compact Vaisman manifold $M^{2n+2}$
is unitary and Killing, we can apply \cite[Theorem~6.13]{tondeur} to $V$.
We obtain a long exact sequence of cohomology groups
\begin{equation*}\label{long}
\dots \to H^k_B(M, V) \xrightarrow{ [\eps_{d\eta}]} H^{k+2}_B (M,V)
\xrightarrow{ [\id]}  H^{k+2}(M) \xrightarrow{[i_V]} H^{k+1}_B (M,V) \to \dots
\end{equation*}
where for a map $f$ between spaces of forms we write $[f]$ for the induced map
in cohomology.  Note that, following the notation in Section \ref{results}, $H^k_B(M, V)$ is the basic cohomology of $M$ associated with the vector field $V$.

Now we apply Theorem~\ref{basic-cohomologies} in  terms of the type $H^k_B(M, V)$ in the above sequence, with $\fol=\langle V \rangle$ and $W=U$.  Moreover, we use  Corollary~\ref{basic-cohomology-DeRham-cohomology} for the terms of the form $H^{k}(M)$, again with $W=U$.
Since $[i_U, i_V] =0$ and $[i_U, \eps_{d\eta}] = \eps_{i_U d\eta} =0$, we get
that the following diagram is commutative

\begin{equation*}
\label{double}
\begin{gathered}
\xymatrix@C4.4em{
H^{k}_B(M, V) \ar[r]^-{[\eps_{d\eta}]} \ar@{<-}[d]_-{[(\id,\eps_\omega)]}^-{\cong}  & H^{k+2}_{B} (M, V)
\ar[r]^-{[\id]} \ar@{<-}[d]_-{[(\id,\eps_\omega)]}^-{\cong}  & H^{k+2}(M)
\ar@{<-}[d]_-{[(\id,\eps_\omega)]}^-{\cong} \ar[r]^-{[i_V]} &
H^{k+1}_B(M, V)
\ar@{<-}[d]_-{[(\id,\eps_\omega)]}^-{\cong}\\
{\begin{smallmatrix}
H^k_B(M,\left\langle U,V \right\rangle)\\
\oplus \\
H^{k-1}_B(M,\left\langle U,V \right\rangle)
\end{smallmatrix}}
\ar@{.>}[r]^-{
\left(
\begin{smallmatrix}
[\eps_{d\eta}] & 0 \\ 0& [\eps_{d\eta}]
\end{smallmatrix}
 \right)
} &
{\begin{smallmatrix} H^{k+2}_B(M,\left\langle U,V \right\rangle)\\ \oplus \\
H^{k+1}_B(M,\left\langle U,V \right\rangle)
\end{smallmatrix}}
\ar@{.>}[r]^-{
\left(
\begin{smallmatrix}
[\id] & 0 \\ 0& [\id]
\end{smallmatrix}
 \right)
}  &
{\begin{smallmatrix} H^{k+2}_B(M,U)\\ \oplus \\
H^{k+1}_B(M,U)
\end{smallmatrix}}
\ar@{.>}[r]^-{\left(
\begin{smallmatrix}
[i_V] & 0 \\ 0& [i_V]
\end{smallmatrix}
 \right)} &
{\begin{smallmatrix} H^{k+1}_B(M,\left\langle U,V \right\rangle)\\ \oplus \\
H^{k}_B(M,\left\langle U,V \right\rangle)
\end{smallmatrix}}
\\
 }
\end{gathered}
\end{equation*}

Since the vertical maps in the above diagram are isomorphisms and the upper sequence is exact, we get that the lower sequence is also exact.
Moreover, note that the bottom sequence splits in two copies of
\begin{equation}
\label{basic}
\dots \to H^k_B(M, \left\langle U,V \right\rangle) \xrightarrow{ [\eps_{d\eta}]}
H^{k+2}_B (M,\left\langle U,V \right\rangle)
\xrightarrow{ [\id]}  H^{k+2}_B(M,U) \xrightarrow{[i_V]} H^{k+1}_B (M,\left\langle
U,V \right\rangle) \to \dots
\end{equation}
Since every Vaisman manifold is transversely K\"ahler with respect to the 2-dimensional foliation defined by the Lee and { anti-Lee} vector fields,
we can apply  El Kacimi-Alaoui results \cite{elkacimi}.  In fact, he proves a transversely Hard Lefschetz theorem
in cohomology \cite[page 97]{elkacimi} which in our case gives that
\begin{equation}
\label{THLT}
\lef_{k}^{UV}:=[\eps_{d\eta}]^{n-k} \colon H^k_B(M,\left\langle U,V \right\rangle) \to
H^{2n-k}_B(M,\left\langle U,V \right\rangle)
\end{equation}
 is an isomorphism, for all  $0\le k \le n$.
Therefore,  the map
\begin{equation*}
[\eps_{d\eta}] \colon H^{k-1}_B(M,\left\langle U,V \right\rangle) \to
H^{k+1}_B(M,\left\langle U,V \right\rangle)
\end{equation*}
is injective, for all $1\le k\le n$.

Now, consider the following part of the long exact sequence \eqref{basic}
\begin{equation*}\label{partlong}
H^{k}_B (M,\left\langle U,V \right\rangle)
\xrightarrow{ [\id]}
H^{k}_B (M,U)
\xrightarrow{[i_V]} H^{k-1}_B(M,\left\langle U,V \right\rangle)
\xrightarrow{ [\eps_{d\eta}]} H^{k+1}_B (M,\left\langle U,V \right\rangle)
\to \dots
\end{equation*}
For $1\le k \le n$, since $[\eps_{d\eta}]$ is injective we get that $[i_V]$ is zero
and hence the map
\begin{equation*}
H^{k}_B (M,\left\langle U,V \right\rangle)
\xrightarrow{ [\id]}
H^{k}_B (M,U)
\end{equation*}
is surjective.
In particular, for $k=1$ we have that
\begin{equation}\label{isomorphismuno}
[Id]: H^1_B(M, \left\langle U, V \right\rangle) \rightarrow H^1_B(M, U)
\end{equation}
is an isomorphism.
Moreover, taking also into account the injectivity of $[\eps_{d\eta}]$, from  \eqref{basic}   we obtain the following short exact sequence
\begin{equation}\label{short}
0  \to H^{k-2}_B(M, \left\langle U,V \right\rangle) \xrightarrow{ [\eps_{d\eta}]}
H^{k}_B (M,\left\langle U,V \right\rangle) \xrightarrow{ [\id]}  H^{k}_B(M,U) \to 0
\end{equation}
for $2\le k\le n$.

The transversal Hard Lefschetz theorem also implies that the map
\begin{equation*}
[\eps_{d\eta}] \colon H^{2n-k-1}_B(M,\left\langle U,V \right\rangle) \to
H^{2n-k+1}_B(M,\left\langle U,V \right\rangle)
\end{equation*}
is surjective, for each $1\le k\le n$.
We consider the following part of the long exact sequence \eqref{basic}
\begin{equation}\label{partlong2}
H^{2n-k-1}_B(M, \left\langle U,V \right\rangle) \xrightarrow{[\eps_{d\eta}]} H^{2n-k+1}_B(M,\left\langle U,V \right\rangle)
\xrightarrow{[\id]} H^{2n-k+1}_B(M,U)\xrightarrow{[i_V]}H^{2n-k}_B(M, \left\langle U,V \right\rangle)
\end{equation}
for $1\le k\le n$. As $[\eps_{d\eta}]$ is surjective, we get that
$$[\id]:H^{2n-k+1}_B(M,\left\langle U,V \right\rangle)\longrightarrow H^{2n-k+1}_B(M,U)$$
is a zero map. Hence
\[
[i_V]:H^{2n-k+1}_B(M,U)\longrightarrow H^{2n-k}_B(M,\left\langle U,V \right\rangle)
\]
is injective. In particular, for $k=1$, the map
{
\begin{equation}\label{2n-isomorphism}
[i_V]:H^{2n}_B(M,U)\longrightarrow H^{2n-1}_B(M,\left\langle U,V \right\rangle)
\end{equation}
is an isomorphism since $H^{2n+1}_B(M, \left\langle U, V \right\rangle ) = \{ 0 \}$}.
For $2\le k\le n$, we get an exact sequence
\begin{equation}\label{short2}
0  \to H^{2n-k+1}_B (M,U) \xrightarrow{ [i_V]}
H^{2n-k}_B (M, \left\langle U,V \right\rangle) \xrightarrow{ [\eps_{d\eta}]}  H^{2n-k+2}_B(M,\left\langle U,V \right\rangle) \to 0.
\end{equation}

Using the transversal Hard Lefschetz isomorphism \eqref{THLT}, we can zip the exact sequences \eqref{short} and \eqref{short2}
like in the following diagram
\begin{small} 
\begin{equation} \label{magic}
\begin{gathered}
\xymatrix@C1em{
0 \ar[rr] && H^{k-2}_B(M,\left\langle U,V \right\rangle)
\ar[rr]^-{[\eps_{d\eta}]} \ar[d]_{\lef_{k-2}^{UV}}^{\cong}
&& H^{k}_{B} (M,\left\langle U,V \right\rangle)
\ar[rr]^-{[\id]} \ar[d]^{\lef_{k}^{UV}}_{\cong}
&& H^k_B(M,U) \ar[r] & 0 \\
0 && H^{2n -k+2}_B (M, \left\langle U,V \right\rangle) \ar[ll] &&
H^{2n-k}_B (M,\left\langle U,V \right\rangle) \ar[ll]_-{[\eps_{d\eta}]}  &&
H^{2n-k+1}_B (M,U) \ar[ll]_-{[i_V]} \ar@{.>}[u]^{T_k} & 0 \ar[l]   \\
 }
\end{gathered}
\end{equation}
\end{small}
where the operator $T_k$ in the diagram is defined as
\begin{align*}
T_k:&=[\id]\circ  (\lef_{k}^{UV})^{-1}\circ [i_V], \mbox{ for } 2\leq k \leq n.
\end{align*}
Note that $T_k$, with $k = 0, 1$, may be defined in a similar way, that is,
\begin{equation}\label{defi01}
T_k:=[\id]\circ  (\lef_{k}^{UV})^{-1}\circ [i_V], \mbox{ for } k = 0, 1.
\end{equation}
\noindent {\bf Second step.} We will see that the linear map
\[
T_k: H^{2n+1-k}_B(M, U) \rightarrow H^{k}_B(M, U), \mbox{ with } 0 \leq k \leq n,
\]
is an isomorphism  and that the inverse morphism of $T_k$ is just the $U$-basic Lefschetz isomorphism in degree $k$
\[
\lef_k^{U}: H^k_B(M, U) \to H^{2n+1-k}_B(M, U).
\]
For $k=1,$ using \eqref{defi01}  we  have that $T_1$ is an isomorphism, as the maps in \eqref{isomorphismuno} and \eqref{2n-isomorphism} are isomorphisms.

On the other hand, from Theorem \ref{basic-cohomologies} and Corollary \ref{basic-cohomology-DeRham-cohomology}, we deduce that the maps
\[
[\eps_w]: H^{2n+1}_B(M, U) \to H^{2n+2}(M), \; \; \; [\eps_w]: H^{2n}_B(M, \left\langle U, V \right\rangle) \to H^{2n+1}_B(M, V)
\]
are isomorphisms and that the inverse morphisms are
\[
[i_U]: H^{2n+2}(M) \to H^{2n+1}_B(M, U), \; \; \; [i_U]: H^{2n+1}_B(M, V) \to H^{2n}_B(M, \left\langle U, V \right\rangle).
\]
Thus, since $H^{2n+2}(M) = \left\langle [\omega \wedge \eta \wedge (d\eta)^n] \right\rangle$, we have that
$H^{2n+1}_B(M, U) = \left\langle[\eta \wedge (d\eta)^n]\right\rangle$, Further, we have
$H^{2n}_B(M, \left\langle U, V \right\rangle) = \left\langle [(d\eta)^n] \right\rangle$ and the map
\begin{equation}\label{i-V-2n+1}
[i_V]: H^{2n+1}_B(M, U) \to H^{2n}_B(M, \left\langle U, V \right\rangle)
\end{equation}
is an isomorphism.
Moreover, it is clear that the map
\begin{equation}\label{Id-0}
[\id]: H^{0}_B(M, \left\langle U, V \right\rangle) \to H^{0}_B(M, U)
\end{equation}
also is an isomorphism.
Therefore, from \eqref{defi01} we deduce that $T_0$  is an isomorphism.

Next, we will show that $T_k$ is an isomorphism of vector spaces for $2\le k\le n$. 

Let us start by proving the injectivity of  $T_k$.  Let $x\in H^{2n-k+1}_B (M,U)$ be such that $[\id]\circ (\lef_{k}^{UV})^{-1} \circ [i_V]x=0$.
Since the upper sequence is exact, there exists
$w\in H^{k-2}_B(M,\left\langle U,V \right\rangle)$ such that
\[
[\eps_{d\eta}]w=(\lef_{k}^{UV})^{-1} \circ [i_V]x=0.
\]
We also have $[\eps_{d\eta}] [i_V]x=0$.  Hence,
\[
\lef_{k-2}^{UV}w=[\eps_{d\eta}] \circ \lef_{k}^{UV} \circ [\eps_{d\eta}] w= [\eps_{d\eta}] \circ [i_V]x=0.
\]
As $\lef_{k-2}^{UV}$ is an isomorphism, we get that $w=0$ and therefore $(\lef_{k}^{UV})^{-1} \circ [i_V]x=0$.
Since $(\lef_{k}^{UV})^{-1}$ is an isomorphism, we get  $[i_V]x=0$.
As $[i_V]$ is injective we conclude that  $x=0$.

Now we will check surjectivity. Let $w\in H^{k}_B (M,U)$. Then there exists
$z\in H^{k}_B (M, \left\langle U,V \right\rangle)$ such that ${[\id]z=w}$. Define
\begin{equation}\label{eq:zprime}
z'=z- [\eps_{d\eta}] \circ (\lef_{k-2}^{UV})^{-1} \circ [\eps_{d\eta}] \circ\lef_{k}^{UV} z.
\end{equation}
As  $[\id] \circ [\eps_{d\eta}]=0$, we get that
\[
[\id]z'=[\id]z=w.
\]
Further, we multiply \eqref{eq:zprime} by $[\eps_{d\eta}] \circ\lef_{k}^{UV}$ on the left and we use that
$[\eps_{d\eta}] \circ \lef_{k}^{UV}\circ [\eps_{d\eta}] = \lef_{k-2}^{UV}$. We get
\[
[\eps_{d\eta}] \circ\lef_{k}^{UV} z'=[\eps_{d\eta}] \circ\lef_{k}^{UV}z -  [\eps_{d\eta}] \circ\lef_{k}^{UV} z=0.
\]
Thus $(\lef_{k}^{UV})z'$ is in the kernel of $[\eps_{d\eta}]$ and therefore the exactness of the bottom sequence
implies that there is $x\in H^{2n-k+1}_B (M,U)$ such that $[i_V]x=(\lef_{k}^{UV})z'$.  Then,
\[
Tx=[\id] \circ(\lef_{k}^{UV})^{-1} \circ [i_V]x= [\id] z'=w.
\]
We conclude that $T$ is also surjective.

Next, we will show that property $(i)$ in the claim of Theorem~\ref{basic-theorem} holds. 

Let $[\beta]_{U} \in H^{k}_{B}(M,U)$. Then, due to the surjectivity of $[\id]$ in \eqref{magic},
there is $z\in H^{k}_B(M,\left\langle U,V \right\rangle)$ such that $[\id]z=[\beta]_{U}$. Let
$\beta''\in \Omega^{k}_B(M,\left\langle U,V \right\rangle)$
 be the $\left\langle U,V \right\rangle$-harmonic form in the cohomology class $z$.
Then
\begin{equation*}
[\beta'']_{U}=  [\id][\beta'']_{UV} =[\id]z=[\beta]_{U},
\end{equation*}
 where $[\beta'']_{UV}$ is the cohomology class in $H^k_B(M, \left\langle U, V \right\rangle )$ induced by $\beta''$. Moreover, using a result in \cite{elkacimi} (see \cite[Proposition 3.4.5]{elkacimi}), we have that the operator $\eps_{d\eta}$ commutes with the $\left\langle U, V \right\rangle$-basic Laplacian and, so, $\eps_{d\eta}$ sends $\left\langle U, V \right\rangle$-basic-harmonic forms into $\left\langle U, V \right\rangle$-basic-harmonic forms. Thus,
\[
\eps_{d\eta}^{n-k+1}\beta''\in \Omega^{2n-k+2}_{\Delta_\fol}(M,\left\langle U,V \right\rangle),
\]
 and, from \eqref{THLT}, we have that
\[
(\eps_{d\eta})^{n-l}: \Omega^{l}_{\Delta_\fol}(M, \left\langle U, V \right\rangle ) \to \Omega^{2n-l}_{\Delta_\fol}(M, \left\langle U, V \right\rangle )
\]
is an isomorphism for $0 \leq l \leq n$.
This implies that there exists
$\gamma\in \Omega^{k-2}_{\Delta_\fol}(M,\left\langle U,V \right\rangle)$ such that 
\begin{equation*}
\eps_{d\eta}^{n-k+2}\gamma=\eps_{d\eta}^{n-k+1}\beta''.
\end{equation*}
So,
\begin{equation*}
L^{n-k+1}(\beta''-\eps_{d\eta}\gamma)=\eps_{d\eta}^{n-k+1}(\beta''-\eps_{d\eta}\gamma)=0.
\end{equation*}
Next, we check that $\beta':=\beta''-\eps_{d\eta}\gamma$ has the required properties in $(i)$.
Indeed, we have
\begin{equation*}
i_V(\beta''-\eps_{d\eta}\gamma)=0,
\end{equation*}
as $\beta''\in\Omega^{k}_B(M,\left\langle U,V \right\rangle)$, $i_V d\eta=0$  and $\gamma \in\Omega^{k-2}_B(M,\left\langle U,V \right\rangle)$.
Moreover, we have
\begin{equation}\label{eq:betaclass}
[\beta''-\eps_{d\eta}\gamma]_{U}=[\beta'']_{U}.
\end{equation}
In fact
$\eps_{d\eta}\gamma=d(\eta\wedge\gamma)$,
as $\gamma$ is closed. Moreover $\eta\wedge\gamma\in\Omega^{k}_B(M,U)$, as $i_U(\eta\wedge\gamma)=0$ and
\begin{equation*}
\lie_U(\eta\wedge\gamma)=i_U d(\eta\wedge\gamma)= i_U (d\eta\wedge\gamma)=0,
\end{equation*}
since $i_U d\eta=0$ and $i_U \gamma=0$. Thus $[\eps_{d\eta}\gamma]_{U}=0$ and \eqref{eq:betaclass} holds.

Now, we consider  the map
\[
\lef_k^{U}: H^{k}_{B}(M,U) \longrightarrow H^{2n+1-k}_{B}(M,U)
\]
defined as the inverse of the isomorphism $T_k$, that is, $\lef_k^{U}=T_k^{-1}$.
We will show that $\lef_k^{U}$ can be computed as in $(ii)$:
if $\beta' \in [\beta]_{U}$ satisfies the conditions in $(i)$, then
\[
\lef_{k}^{U}[\beta]_{U} = [\epsilon_{\eta}L^{n-k}\beta']_{U}.
\]
First of all, condition $(i)$ implies
$$d(\epsilon_{\eta}L^{n-k}\beta')=L^{n-k+1}\beta'=0.$$
Next,
\begin{equation*}
i_U(\epsilon_{\eta}L^{n-k}\beta')=0,
\end{equation*}
as $i_U \eta=0$, $i_U d\eta=0$ and $\beta'$ is $U$-basic.
We conclude that $\epsilon_{\eta}L^{n-k}\beta'$ is a closed form in $\Omega^{2n+1-k}_B(M,U)$.
Moreover,  from (\ref{THLT}),
\begin{align*}
T_k[\epsilon_{\eta}L^{n-k}\beta']_{U}&= [\id]\circ  (\lef_{k}^{UV})^{-1}\circ [i_V] [\epsilon_{\eta}L^{n-k}\beta']_{U}\\
                                    &= [\id]\circ  (\lef_{k}^{UV})^{-1} [\eps_{d\eta}^{n-k} \beta']_{UV}\\
                                      &=  [\id][\beta']_{UV}=[\beta']_{U}=[\beta]_{U}.
\end{align*}
Therefore,
\[
T_k^{-1}[\beta]_{U} = [\epsilon_{\eta}L^{n-k}\beta']_{U}=\lef_k^{U}[\beta]_{U}.
\]
\end{proof}


\section{Proof of Theorem \ref{basic-Hard-Lef-property} }\label{Hard-Lefschetz-property}
\label{proof14}

\begin{proof}
{\bf Proof of $(1)$ in Theorem \ref{basic-Hard-Lef-property}}

\medskip

\noindent{\em (Lefschetz $\Rightarrow $ basic Lefschetz)}

Assume that the l.c.s. structure $(\omega, \eta)$ is Lefschetz and denote by
\[
[i]: H^k_{B}(M, U) \longrightarrow H^{k}(M)
\]
the canonical monomorphism, that is,
$[i] = [(Id, \epsilon_{\omega})]|_{H^k_{B}(M, U) \oplus \{0\}}$ (cf.  Corollary \ref{basic-cohomology-DeRham-cohomology}).

Let $\beta$ be a closed basic $k$-form on $M$, with $k \leq n$. Then, there exists
a closed $k$-form $\beta'$ on $M$ such that $\beta' \in [\beta]$ and
\begin{equation}\label{proper-beta-prima}
{\mathcal L}_U\beta' = 0, \; \; i_V\beta' = 0, \; \; L^{n-k+2}\beta' = 0, \; \; L^{n-k+1}\epsilon_{\omega}\beta' = 0.
\end{equation}
Let $\bar{\beta} = G\delta (\beta'-\beta) $, where $G$ is the Green operator
on $M$. Then $\lie_U \bar{\beta}=0$
since ${\mathcal L}_{U}\beta =0$ and $\lie_U$ commutes with $G$ and $\delta$, as $U$ is a Killing vector field.

From the properties of the Green operator, it follows that
\[
\beta' = \beta + d \bar{\beta}.
\]
Therefore, if we take
\begin{equation}\label{beta-1-prima}
\beta'' = \beta' - \omega \wedge i_{U}\beta'
\end{equation}
then $d \beta'' = 0$, $\beta'' \in \Omega^k_{B}(M, U)$ and
\[
\beta'' = \beta + d(\bar{\beta} - \omega \wedge i_U\bar{\beta}).
\]
This implies that
 \[
\beta'' \in [\beta]_U.
\]
Moreover, from (\ref{proper-beta-prima}) and (\ref{beta-1-prima}), it follows that
\[
i_V\beta'' = 0, \; \; \; L^{n-k+1}\beta'' =0.
\]
This shows that every class  $[\beta]_U$ contains a representative $\beta''$ that
satisfies~\eqref{basic-conditions}.

We define $\lef_k^U: H^k_B(M, U) \longrightarrow H^{2n+1-k}_B(M, U)$ by
 \begin{equation}\label{Lef-k-B}
\lef_k^U[\beta]_U = [\epsilon_\eta L^{n-k}\beta']_U,
\end{equation}
where $\beta'$ is any representative of  $[\beta]_U$ that
satisfies~\eqref{basic-conditions}.
Note that $\epseta L^{n-k}\beta'$ is closed as $d(\epseta L^{n-k} \beta') =
L^{n-k+1} \beta' =0$ and basic because
\begin{equation*}
	i_U \epseta L^{n-k} \beta' = \epseta L^{n-k} i_U \beta' =0.
\end{equation*}
We must prove that $\lef_k^U$ is well defined and is an isomorphism.
To show that $\lef_k^U$ is well defined, we have to check that for any
 $\beta'$, $\beta''\in [\beta]_U$ that satisfy~\eqref{basic-conditions}, one
has
 \begin{equation*}
	[\eps_\eta L^{n-k}\beta']_U = [\eps_\eta L^{n-k} \beta'']_U.
\end{equation*}
Since $\beta'$ and $\beta''$
are basic and
\begin{align*}
	i_V \beta' &=0, & L^{n-k+1}\beta' & =0, & i_V\beta''&=0, &
	L^{n-k+1}\beta''&=0,
\end{align*}
it is easy to check that $\beta'$ and $\beta''$
satisfy~\eqref{proper-beta-prima}. Moreover, they are representatives of
$[\beta]\in H^k(M)$.
Thus by the defining properties of $\lef_k$, we have
\begin{equation}
	\label{epseta}
	[\eps_\eta L^{n-k}\beta'] = \lef_k[\beta] = [\eps_\eta L^{n-k} \beta''].
\end{equation}
Thus  $[i][\eps_\eta L^{n-k}\beta']_U = [i][\eps_\eta L^{n-k}\beta'']_U$.
Since $[i]$ is a monomorphism, we get \allowbreak  $[\eps_\eta L^{n-k}\beta']_U = [\eps_\eta
L^{n-k}\beta'']_U$. Hence $\lef_k^U$ is well defined and satisfies the required
property for its computation.

From~\eqref{epseta}, we get that
 \begin{equation*}
\lef_k \circ\, [i][\beta]_U = [i] \circ  \lef_k^U [\beta]_U
\end{equation*}
for any  $[\beta]_U\in H^k_B(M, U)$. Thus
if $\lef_k^U[\beta]_U =0$ then $\lef_k [i][\beta]_U =0$. As $\lef_k$ is an
isomorphism, we get that $[i][\beta]_U =0$. Since $[i]$ is
monomorphism this implies that  $[\beta]_U=0$, that is $\lef_k^U$ is injective.

By Poincar\'e duality for basic cohomology \cite[Theorem~7.54]{tondeur}
(see also calculations on page 69 of~\cite{tondeur}), it follows that
\begin{equation*}
	\dim H^k_B(M, U) = \dim H^{2n+1-k}_B(M, U).
\end{equation*}
Thus  $\lef^k_U$ is an isomorphism, being an injective map between two vector
spaces of the same dimension.

\medskip

\noindent ({\em Basic Lefschetz $\Rightarrow $ Lefschetz})

Assume that the l.c.s. structure on $M$ is basic Lefschetz.

Let $\gamma$ be a closed $k$-form on $M$, with $k \leq n$.
 We have to show that there is $\gamma'\in [\gamma]$ satisfying
\eqref{proper-beta-prima}.
Note that
\[
[(Id, \eps_\omega) ]^{-1}[\gamma] = ( [\gamma_1]_U, [\gamma_2]_U)\in \Omega^k_B(M, U)\oplus
\Omega^{k-1}_B(M, U)
\]
for some basic forms $\gamma_1$ and $\gamma_2$ on $M$.
Since the basic Lefschetz property holds, we can choose
 $\gamma_1' \in [\gamma_1]_U$ and $\gamma_2'\in [\gamma_2]_U$ such that
\[
   i_V\gamma_1' = 0, \; \;  L^{n-k+1}\gamma_1' = 0,
\]
and
\[
 i_V\gamma_2' = 0, \; \;  L^{n-k+2}\gamma_2' = 0.
\]
Thus, we can consider the closed $k$-form
\[
\gamma' = \gamma_1' + \omega \wedge \gamma_2'
\]
and it is easy to see that $\gamma'\in [\gamma]$ and
\[
 {\mathcal L}_{U}\gamma' = 0, \; \; i_V\gamma' = 0, \; \; L^{n-k+2}\gamma' =
 0,\; \;  L^{n-k+1}\epsilon_\omega \gamma' = 0.
\]

%
%
%
Now, consider the following isomorphisms
\[
[ (Id,\eps_\omega) ]^{-1}: H^{k}_{}(M) \longrightarrow H^{k}_{B}(M, U) \oplus H^{k-1}_{B}(M, U),
\]

\vspace{-5pt}

\[
 \lef_{k}^{U} \oplus \lef_{k-1}^{U}: H^{k}_{B}(M, U) \oplus H^{k-1}_{B}(M, U) \longrightarrow H^{2n+1-k}_{B}(M, U) \oplus H^{2n+2-k}_{B}(M, U),
\]

\vspace{-5pt}

\[
\sigma: H^{2n+1-k}_{B}(M, U) \oplus H^{2n+2-k}_{B}(M, U) \longrightarrow H^{2n+2-k}_{B}(M, U) \oplus H^{2n+1-k}_{B}(M, U),
\]

\vspace{-5pt}

\begin{equation*}
[ (Id,\eps_\omega)]: H^{2n+2-k}_{B}(M, U) \oplus H^{2n+1-k}_{B}(M, U) \longrightarrow H^{2n+2-k}_{}(M),
\end{equation*}
where $\sigma$ is the canonical involution and $ \lef_r^U: H^r_B(M, U) \longrightarrow H^{2n+1-r}_B(M, U)$ is the isomorphism whose graph is the basic Lefschetz relation $R_{\lef_r}^B$.
We define the map $\lef_k\colon H^k(M) \longrightarrow H^{2n+2-k}(M)$ by
\[
\lef_k = [(Id,\eps_\omega) ] \circ \sigma \circ ( \lef_{k}^{U} \oplus
\lef_{k-1}^{U}) \circ [(Id,\eps_\omega)]^{-1}.
\]
It is straightforward that $\lef_k$ is an isomorphism. It is left to show that
for any closed form $\gamma\in \Omega^k(M)$ that satisfies \eqref{proper-beta-prima}, we
have
\begin{equation}\label{Lef-k-gamma}
\lef_k[\gamma] = [\epsilon_\eta L^{n-k}(Li_U\gamma - \epsilon_{\omega}\gamma)].
\end{equation}
Let $\gamma_1= i_U(\omega\wedge \gamma)$ and $\gamma_2=i_U\gamma$.
Note that the forms $\gamma_1$ and $\gamma_2$ are closed and obviously $i_U\gamma_1 =
i_U\gamma_2=0$. Thus they represent certain classes in $H^*_B(M, U)$. We have
\begin{equation*}
	[(Id,\eps_\omega)]^{-1}[\gamma] = ( [\gamma_1]_U,
	[\gamma_2]_U).
\end{equation*}
Moreover,
$\gamma_1$ satisfies the conditions
\begin{equation*}
	i_V \gamma_1 =0,\  \ L^{n-k+1} \gamma_1=0,\\
\end{equation*}
and $\gamma_2$ satisfies
\begin{equation*}
		i_V \gamma_2 =0,\  \ L^{n-k+2} \gamma_2=0.
\end{equation*}
Thus
\begin{equation}
	\label{first}
	 \lef_k^U[\gamma_1]_U = [\eps_\eta L^{n-k}\gamma_1]_U = [\eps_\eta L^{n-k}
i_U(\omega\wedge \gamma)]_U
\end{equation}
	and
	\begin{equation}	
		\label{second}
		 \lef_{k-1}^U [\gamma_2]_U = [\eps_\eta
		L^{n-k+1} \gamma_2]_U = [\eps_\eta L^{n-k+1} i_U \gamma]_U.
	\end{equation}
Now, we get by definition of the map $\lef_k$ that
\begin{align*}
	\lef_k[\gamma]& = [\eps_\eta L^{n-k+1}i_U \gamma + \eps_\omega \eps_\eta
	L^{n-k} i_U(\omega \wedge\gamma)]
	=[\epsilon_\eta L^{n-k}(Li_U\gamma - \epsilon_{\omega}\gamma)].
\end{align*}
This proves the first part of Theorem~\ref{basic-Hard-Lef-property}.
	\medskip

\noindent{\bf Proof of $(2)$ in Theorem \ref{basic-Hard-Lef-property}}

Assume that the l.c.s. structure $(\omega, \eta)$ is Lefschetz (or,
equivalently, basic Lefschetz)
and denote by
\[
\lef_k^U: H^k_B(M, U) \longrightarrow H^{2n+1-k}_{B}(M, U), \; \; 1 \leq k \leq n,
\]
the isomorphism whose graph is the basic Lefschetz relation $R_{\lef_k}^B$.
Since $U$ is a parallel vector field, we get that the mean curvature
$\kappa$ of the
 foliation $\left\langle U \right\rangle$ is zero. Moreover, the characteristic form
 of $\left\langle U \right\rangle$ is $\omega$ (see page 69 of~\cite{tondeur}). Therefore, by Theorem~7.54
in~\cite{tondeur}, we have a non-degenerate pairing
\begin{equation*}
	\left\langle \cdot,\cdot \right\rangle\colon H^{2n+1-k}_B(M, U) \otimes
	H^k_B(M, U)\longrightarrow \R
\end{equation*}
given by
\begin{equation}
	 \left\langle [\mu]_U,\left[ \beta \right]_U \right\rangle = \int_M
	\omega\wedge \mu \wedge \beta.
\end{equation}

Since $\lef_k^U$ is an isomorphism, we get that the bilinear form
\begin{equation}
	\psi := \left\langle \cdot,\cdot \right\rangle \circ (\lef_k^U,Id)	
	\label{psi}
\end{equation}
on $H^k_B(M, U)$ is non-degenerate. Let $\alpha$ and $\beta$ be  closed basic $k$-forms on $M$. Then, for any  $\alpha'\in [\alpha]_U$ and
$\beta'\in [\beta]_U$ such that
\begin{align*}
	i_V \alpha' &=0, & L^{n-k+1} \alpha' & = 0,& i_V \beta' & = 0, &
	L^{n-k+1}\beta'& = 0,
\end{align*}
 we get that
 \begin{align*}
	 \psi( [\alpha]_U, [\beta]_U )& = \int_M \omega \wedge \eps_\eta L^{n-k}
	 \alpha'\wedge \beta'\\& = (-1)^{k^2} \int_M \omega \wedge \eps_\eta
	 L^{n-k}\beta' \wedge \alpha' = (-1)^k \psi ( [\beta]_U, [\alpha]_U).
 \end{align*}
 This shows that $\psi$ is symmetric if $k$ is even, and skew-symmetric
 if $k$ is odd.

Consequently, for odd $k$ between $1$ and $n$, $c_k(M) := \dim H^k_B(M, U)$
is even. On the other hand, from Theorem \ref{basic-cohomology-DeRham-cohomology} we deduce that
\[
b_k(M) = c_k(M) + c_{k-1}(M),
\]
which implies that
\[
b_k(M)-b_{k-1}(M) = c_k(M) - c_{k-2}(M).
\]
Since $c_{k}(M)$ and $c_{k-2}(M)$ are both even, we get that $b_k(M)-b_{k-1}(M)$ is also even.
\end{proof}

Finally, we derive Theorem~\ref{main-theorem} from the previously proved
results.
\begin{proof}[(Proof of Theorem \ref{main-theorem}).]
By Theorem~\ref{basic-theorem}, every compact Vaisman manifold satisfies the basic Lefschetz property.
Thus  from 	Theorem~\ref{basic-Hard-Lef-property}, we get that every Vaisman manifold has the Lefschetz property.
\end{proof}
\section{Examples of non Vaisman compact Lefschetz l.c.s. manifolds}
\label{sec:examples}
In this section, we will construct examples of compact l.c.s. manifolds of the first kind which do not satisfy the Lefschetz property. We will also present an example of a compact Lefschetz l.c.s. manifold of the first kind which does not admit compatible Vaisman metrics.
In order to do this we will use the following proposition.
\begin{proposition}\label{lcs-regular}
	Let $M$ be a $(2n+2)$-dimensional compact l.c.s. manifold of the first
	kind such that the space of orbits $N$  of the Lee vector field
	$U$ of $M$ is a quotient manifold.
	\begin{enumerate}[(i)]
		\item If $\pi\colon M \longrightarrow N$ is the canonical projection then
			there exists a contact $1$-form $\eta_N$ on $N$
			 such that $\pi^* \eta_N= \eta$, $\eta$ being the anti-Lee $1$-form of $M$. Moreover, the anti-Lee vector field $V$ of $M$ is $\pi$-projectable and its
			 projection $\xi\in \vect(N)$
			 is the Reeb vector field of the contact manifold
			 $(N,\eta_N)$.
\item There exists a Riemannian metric $g$ on $M$ such that $U$
	is parallel and unitary with respect to $g$, and, in addition,
	\begin{equation*}
		\omega(X) = g(X,U),\ \mbox{ for all } X\in \vect(M),
	\end{equation*}
	where $\omega$ is the Lee $1$-form of $M$.
	\end{enumerate}
\end{proposition}
\begin{proof}
	The first claim is well known and can be found in \cite{Va}. 	
	
We will prove $(ii)$.
Denote by $\fol$ the 	vector subbundle induced by the foliation $\omega=0$. Then $TM =
	\fol\oplus \left\langle U \right\rangle$.
    From an arbitrary Riemannian metric $g_N$ on $N$, we can define a
	Riemannian metric $g$ on $M$ such that
	\begin{equation*}
		\pi \colon (M,g) \longrightarrow (N,g_N)
	\end{equation*}
	is a Riemannian submersion with the horizontal subbundle $\fol$ and the
	vector field $U$ is unitary with respect to $g$. This implies that
	$\omega (X) = g(X,U)$. Moreover, it is easy to prove that $U$ is a
	Killing vector field. Thus, since the dual $1$-form to $U$ with respect
	to $g$ is $\omega$ and it is closed, we conclude that $U$
	is parallel.
\end{proof}

Under the same conditions as in Proposition~\ref{lcs-regular}, it is clear that the basic Lefschetz property for the l.c.s. manifold $M$ is equivalent to the Lefschetz property for the base contact manifold $N$.  We recall that on a contact manifold of dimension $2n+1$ with contact structure $\eta$ and Reeb vector field $\xi$, we may define the Lefschetz relation between the de Rham cohomology groups $H^{k}(N)$ and $H^{2n+1-k}(N)$, for $0 \leq k \leq n$, by
\[
R_{\lef_k} = \{ [\beta], [\eps_\eta L^{n-k}\beta] ) \;|\; \beta \in \Omega^k(N), d\beta = 0, i_{\xi}\beta = 0, L^{n-k+1}\beta = 0 \},
\]
where $L$ is the operator $\eps_{d\eta}$.
Then, $\eta$ is said to be a contact Lefschetz structure if the relation $R_{\lef_k}$ is the graph of an isomorphism $\lef_k: H^{k}(N) \to H^{2n+1-k}(N)$ (see \cite{CaNiYu}).

 So, using Theorem \ref{basic-Hard-Lef-property} and Proposition~\ref{lcs-regular}, we deduce the following result.
\begin{corollary}\label{examples}
Let $M$ be a $(2n+2)$-dimensional compact l.c.s. manifold of the first
	kind such that the space of orbits of the Lee vector field
is the contact manifold $N$. Then, the following conditions are equivalent:
\begin{enumerate}
\item
The l.c.s. structure on $M$ satisfies the Lefschetz property.

\item
The l.c.s. structure on $M$ satisfies the basic Lefschetz property.

\item
The contact structure on $N$ satisfies the Lefschetz property.
\end{enumerate}
\end{corollary}
Now, let $N$ be a compact contact manifold and consider  the product manifold $M= N \times S^1$ with its  standard l.c.s. structure of the first kind  (see Section \ref{Compact-Sasakian-Vaisman}). Then, it is clear that the space of orbits of the Lee vector field of $M$ is $N$. Thus, using Corollary \ref{examples} and taking as $N$ the examples of non-Lefschetz compact contact manifolds considered in \cite{CaNiMaYu}, we obtain  examples of compact l.c.s. manifolds of the first kind which satisfy the following conditions:
\begin{enumerate}
\item
Their Betti numbers satisfy relations (\ref{Betti-numbers}) in Theorem \ref{basic-Hard-Lef-property}.

\item
They do not satisfy neither the Lefschetz property nor the basic Lefschetz property (and, therefore, they do not admit compatible Vaisman metrics).
\end{enumerate}
Note that $(1)$ follows using that $b_k(N)$ is even if $k$ is odd and $k \leq n$, with $\dim N = 2n+1$ and $b_k(N)$ the $k$-th Betti number of $N$.

On the other hand, in \cite{CaNiMaYu2}, we present an example of a compact Lefschetz contact manifold $N$ which does not admit any Sasakian structure. So, the standard l.c.s. structure of the first kind on $M = N \times S^1$ is Lefschetz and basic Lefschetz. However, $M$ does not admit compatible Vaisman metrics.

We conclude stating an open problem concerning these topics.
It would be interesting  to find examples of compact l.c.s. manifolds of the first kind which satisfy the basic Lefschetz property but they do not satisfy the Lefschetz property and vice versa.


\end{document}